\documentclass[12pt]{article}
\usepackage[all]{xy}
\usepackage{amsmath, amssymb,euscript,graphicx,amsthm,enumerate}
\input hart.sty   
\usepackage{amsmath, amssymb,amsthm, fullpage, dcpic, pictexwd, euscript, graphicx, enumerate}

\begin{document}
%%%%%%%%%%%%%%%%%%%%%%%%%
\title{Families of Log Canonically Polarized Varieties}
\author{Ariana Dundon\\ \small\textit{Department of Mathematics, Montgomery College, 51 Mannakee Street, SW36, Rockville, MD 20850}}
%%%%%%%%%%%%%%%%%%%%%%%

\maketitle
\section{Introduction}
Shafarevich's hyperbolicity conjecture states that for $q\geq 2$, $\sM_q$, the moduli stack of curves of genus $q$, is algebraically hyperbolic.  That is, there are no non-constant maps from $A\to \sM_q$ where $A$ is an abelian variety and from $\C\setminus \{0\}\to \sM_q$.  There is a useful interpretation of this result in terms of families, where by \emph{family} we mean a flat, projective morphism with connected fibers.  A family is \emph{isotrivial} if any two generic fibers are isomorphic.  Since non-isotrivial families correspond to non-constant maps into the moduli space, we can restate Shafarevich's hyperbolicity conjecture as follows: Any smooth family of curves of genus $q$ over an abelian variety must be isotrivial and any family of curves of genus $q$ over $\P^1$ is either isotrivial or has at least 3 singular fibers.  This conjecture was confirmed in a special case by Parshin \cite{Parshin68} and in general by Arakelov \cite{Arakelov71}.

These results have produced many interesting generalizations.  For a survey of current results in this area, \cite{Kovacs09b} is an excellent resource.  In particular, in \cite{Kovacs97a} and \cite[0.2]{Kovacs00a}, Kov\'{a}cs shows the algebraic hyperbolicity of $C_h$, the moduli stack of \emph{canonically polarized varieties}, that is, varieties for which the canonically line bundle is ample.  As with families of curves, one piece of this result is that any non-isotrivial family of canonically polarized varieties over $\P^1$ must have at least 3 singular fibers.  

In this paper, we will examine the generalization of this statement to families of varieties with other polarizations; specifically, we look at families of \emph{log canonically polarized varieties}, where the ample line bundle is given by the canonical bundle twisted by an effective divisor. With the addition of certain mild hypotheses, which we will decribe in detail in later sections, we prove an analogous result:

\begin{theorem}\label{main}
Let $g:(Y,D)\to \P^1$ be a family, $\Delta\subset \P^1$ the locus of $D$-singular fibers and $B_0=\P^1\setminus \Delta$.  Suppose that the line bundle $\omega_{Y/B}(D)$ is $g$-ample and that $|\Aut(Y_b)|<\infty$ for $b\in \P^1$.  Suppose also that $D$ is an snc divisor for which each component intersects the fibers of $g$ transversely and avoids the singular locus of the fibers over $\Delta$.  

Then either $g$ is isotrivial or $|\Delta|\geq 3$.  
\end{theorem}

\section{Notation and Definitions}
All varieties will be defined over an algebraically closed field of characteristic $0$.  

First we fix some notation.  $B$ will denote a smooth projective curve and $g:Y\to B$ a flat, projective morphism with connected fibers from a smooth projective variety $Y$ to $B$.  Let $\sL$ be line bundle on $Y$.  For $b\in B$, we denote by $Y_b$ the fiber of the map $g$ over the point $b$, and let $\sL_b$ denote the restriction of the line bundle $\sL$ to $Y_b$.  We call the map $f$ a \emph{family} of pairs $(Y_b,\sL_b)$ over the curve $B$.  Such a family is \emph{isotrivial} if for any two general points $a, b\in B$, there's an isomorphism $\phi:Y_a\to Y_b$ and such that $\phi^*\sL_b\iso \sL_a$. We will consider non-isotrivial families since they correspond to non-constant maps into the moduli stack.  

If the line bundle $\sL_b$ is ample, then it is a polarization for the fiber $Y_b$, and if this is the case for all $b\in B$ we have a family of polarized varieties.  If $\sL_b\iso \omega_{Y_b}$ we have a family of canonically polarized varieties.  Accordingly, it is interesting to study varieties with polarizations of a similar but more general form, those where $\sL_b\iso \omega_{Y_b}(D_b)$ for $D_b$ an effective divisor on the fiber $Y_b$.  Such varieties are known as \emph{log canonically polarized}.  

Consequently, to study families of log canonically polarized varieties, we consider the following situation: Let $g:(Y, \sL)\to B$ be a family of polarized varieties and $D$ an effective divisor on $Y$ such that $\sL\iso \omega_{Y/B}(D)$.  We will also require that $D$ is horizontal with respect to $g$, that is, no irreducible component of $D$ is contained in a fiber of $g$.  This means that $D_b=D|_{Y_b}$ is a divisor on the fiber $Y_b$, and in particular, $\sL_b\iso \omega_{Y_b}(D_b)$.  Thus we will sometimes denote by $g:(Y, D)\to B$ a family along with an effective horizontal divisor $D$ with the understanding that the polarization on a fiber $Y_b$ is the line bundle $\sL_b\iso \omega_{Y_b}(D_b)$. 

The following definition describes what it means for fibers of a family $g:(Y, D)\to B$ to be singular.  

\begin{definition}\label{delta}
Let $\Delta$ be the finite set of points in $B$ over which $\Omega_{Y/B}(\log D)$ is not locally free.  We will refer to $\Delta$ as the locus of \emph{$D$-singular fibers}.  
\end{definition}

The fibers over the points of $\Delta$ are exactly those that are either singular themselves or those that don't meet $D$ transversely, that is, fibers $F$ where the divisor $D+F$ is not simple normal crossing.  

Note however, as following lemma states, that if $D+g^*\Delta$ is snc, then $\Omega_{Y/B}\big(\log (D+g^*\Delta)\big)$ is locally free.  The proof of this lemma is straightforward and so we omit it.  

\begin{lemma}\label{locfree}
Let $g:(Y, D)\to B$ be a family and let $\Delta\subset B$ be the locus of $D$-singular fibers.  Suppose that $D+g^*\Delta$ is snc.  Then $\Omega_{Y/B}\big(\log (D+g^*\Delta)\big)$ is locally free.  
\end{lemma}

Thus we have a short exact sequence of locally free sheaves $$0\to g^*\omega_B(\Delta)\to \Omega_{Y}\big(\log (D+g^*\Delta)\big) \to \Omega_{Y/B}\big(\log (D+g^*\Delta)\big)\to 0$$
which will be useful in the proof of Theorem \ref{van}.

\section{Positivity Results}
We wish to study families $g:(Y, D)\to B$ of log canonically polarized varieties (where the sheaf $\omega_{Y_b}(D_b)$ is ample for $b\in B$).  This is equivalent to saying that the sheaf $\omega_{Y/B}(D)$ is $g$-ample.  

For a non-isotrivial family of canonically polarized varieties, $g:Y\to B$, $g_*\omega_{Y/B}^r$ is ample for some $r>0$ \cite{Kollar87}.  The purpose of this section is to show that for an non-isotrivial family of log canonically polarized varieties, $g_*\omega_{Y/B}(D)^m$ is ample for some $m>0$.

\begin{theorem}\label{det}
Let $g:(Y,D)\to B$ be a non-isotrivial family with $\Delta\subset B$ the locus of $D$-singular fibers.  Suppose that the line bundle $\omega_{Y/B}(D)|_{Y_b}=\omega_{Y_b}(D_b)$ is ample for every $b\in B$.  Suppose further that for $b\in B$, the image of $Y_b$ under the embedding via $\omega_{Y_b}(D_b)^k$ into some $\P^n$ has only finitely many automorphisms that can be written as the restriction of a linear automorphism $\sigma$ of $\P^n$ to the image of $Y_b$.  Then $\det \big(g_*\omega_{Y/B}(D)^n\big)$ is ample on $B$ for $n>0$ appropriately large and divisible.  
\end{theorem}

\begin{proof}
Since $\omega_{Y/B}(D)$ is ample when restricted to any fiber and $g$ is flat, there's a positive integer $m$ such that $\omega_{Y/B}(D)^m$ is very ample when restricted to the fibers of $g$.  Now consider the map $$\delta:\sym^l(g_*\omega_{Y/B}(D)^m)\to g_*\omega_{Y/B}(D)^{ml}$$
For any $b\in B$, we have the map 
$$\delta_b:\sym^l(g_*\omega_{Y/B}(D)^m)\otimes k(b)\to g_*\omega_{Y/B}(D)^{ml}\otimes k(b)$$ 
By base change and cohomology, $$\sym^l(g_*\omega_{Y/B}(D)^m)\otimes k(b)\iso \sym^l(H^0(Y_b, \omega_{Y_b}(D_b)^m))$$ and $$g_*\omega_{Y/B}(D)^{ml}\otimes k(b)\iso H^0(Y_b, \omega_{Y_b}(D_b)^{ml})$$ 
Consider the embedding $\phi_b:Y_b\to \P^r$ via the very ample sheaf $\omega_{Y_b}(D_b)^m$, where 
$$h^0(Y_b, \omega_{Y_b}(D_b)^m)=\rank g_*\omega_{Y/B}(D)^m=r+1$$
We identify $H^0(Y_b, \omega_{Y_b}(D_b)^m)$ with $H^0(\P^r, \sO_{\P^r}(1))$, and since $$\sym^l(H^0(\P^r,\sO_{\P^r}(1))) \iso H^0(\P^r,\sO_{\P^r}(l))$$
we can see that the kernel $K_b$ of $\delta_b$ is exactly the degree $l$ homogeneous polynomials that vanish on the image of $Y_b$ under this embedding.  By choosing $l$ large enough, we can recover the image of the variety $Y_b$ from $K_b$, so fix an $l$ for which this is true for all $b\in B_0$.  

The group $G=\Sl(r+1,k)$ acts on the vector space $H^0(Y_b, \omega_{Y_b}(D_b)^m)$ by changing basis, which corresponds to a different choice of coordinates on $\P^r$.  The kernel $K_b$ is a subspace of some dimension, say $d$, in $V=H^0(\P^r, \sO_{\P^r}(l))$, so it defines a point in the Grassmannian variety $\Gr(d,V)$.  Thus we can see that the $G$-orbit of $K_b$ in $\Gr(d,V)$, denoted by $[K_b]$ does not depend on the choice of basis for $H^0(Y_b, \omega_{Y_b}(D_b)^m)$.  

Now, fix a $b\in B$ and suppose that $[K_b]=[K_{b^\prime}]$ for some $b^\prime \in B$.  Then since $K_b$ recovers the image of the variety $Y_b$, this means that there is some linear change of coordinates $\sigma:\P^r\to \P^r$ such that $\sigma$ takes the image of $Y_{b^\prime}$ to the image of $Y_b$.  In particular, $\sigma|_{Y_{b^\prime}}:Y_{b^\prime}\to Y_b$ is an isomorphism.  Additionally, since $\sigma$ is a linear automorphism, $\sigma^*\sO_{\P^r}(1)\iso \sO_{\P^r}(1)$.  Also, since $Y_{b^\prime}$ and $Y_b$ are embedded in $\P^r$ by the maps $\phi_{b^\prime}$ and $\phi_b$, defined by the line bundles $\omega_{Y_{b^\prime}}(D_{b^\prime})^m$ and $\omega_{Y_b}(D_b)^m$ respectively, we see that $$\begin{array}{rcl}(\sigma|_{Y_{b^\prime}})^*\omega_{Y_b}(D_b)^m & \iso & (\sigma|_{Y_{b^\prime}})^*\phi_b^*\sO_{\P^r}(1)\\
& \iso & \phi_{b^\prime}^*\sigma^*\sO_{\P^r}(1)\\
& \iso & \phi_{b^\prime}^*\sO_{\P^r}(1)\\
& \iso & \omega_{Y_{b^\prime}}(D_{b^\prime})^m\\
\end{array}$$
Since $(\sigma|_{Y_{b^\prime}})$ is an isomorphism, $(\sigma|_{Y_{b^\prime}})^*\omega_{Y_b}\iso \omega_{Y_{b^\prime}}$, so is $(\sigma|_{Y_{b^\prime}})^*\sO_{Y_b}(D_b)^m\iso \sO_{Y_{b^\prime}}(D_{b^\prime})^m$.  The family $(Y,D)\to B$ is non-isotrivial, so there can be only finitely many $b^\prime$'s for which such an isomorphism $\sigma|_{Y_{b^\prime}}$ exists, and hence the set $\{b^\prime\in B\;\:\;[K_{b^\prime}]=[K_b]\}$ is finite. 

Now consider the set of linear automorphisms $\sigma:\P^r\to \P^r$ such that $\sigma(Y_b)=Y_b$.  First note that because $Y_b$ is embedded in $\P^r$, the projectivization of the vector space $H^0(Y_b, \omega_{Y_b}(D_b)^m)$ via $y\mapsto (\psi_1(y):\ldots:\psi_{r+1}(y))$ where $\{\psi_1,\ldots,\psi_{r+1}\}$ is a basis, the image is not contained in any hyperplane of $\P^r$.  Thus if $\sigma$ and $\sigma^\prime$ are both linear automorphisms of $\P^n$ with $\sigma|_{Y_b}=\sigma^\prime|_{Y_b}$, then $\sigma=\sigma^\prime$ since $Y_b$ will contain $n$ linearly independent vectors in $\P^n$.  

Since we assumed that the set of automorphisms of $Y_b$ that come from a linear automorphism of $\P^n$ is finite, then there exists at most finitely many $\sigma\in G=\Sl(r+1,k)$ that take $[K_b]$ to $[K_b]$.  Thus the stablizer of $[K_b]$ in $G$ is $0$-dimensional, and so $\dim G=\dim [K_b]$.  

Thus the kernel of the map $\delta$ has maximal variation, and so by \cite{Kollar90},
cf. \cite[Theorem 4.34]{Viehweg95}, $\det(g_*\omega_{Y/B}(D)^{ml})$ is ample on $B$.  

\end{proof}

Notice that if the fibers of our family have finite automorphism group, as in the case of varieties of general type, then they will satisfy the hypothesis of this theorem.

This theorem tells us that $\det(g_*\omega_{Y/B}(D)^m)$ is ample for some $m$, and we'd like to conclude that $g_*\omega_{Y/B}(D)^m$ is also ample on $B$.  To do this, we must prove several corollaries which are somewhat technical and draw heavily on the methods in \cite[Chapter 2]{Viehweg95}.  We also need to add the hypothesis that $\sO_Y(D)$ is semi-ample, which is a reasonable assumption since we are generalizing the case where $D=0$.    

First, consider the following definition.

\begin{definition}[cf.\protect{\cite[Definition 2.10]{Viehweg95}}]
Let $Y$ be a quasi-projective reduced scheme and $Y_0\subseteq Y$ an open dense subscheme and let $\sG$ be a locally free sheaf on $Y$, of finite, constant rank.  Then $\sG$ is called \emph{weakly positive} over $Y_0$ if:

For an ample invertible sheaf $\sH$ on $Y$ and for any given number $\alpha>0$ there exists some $\beta>0$ such that $\sym^{\alpha\cdot\beta}(\sG)\otimes \sH^\beta$ is globally generated over $Y_0$.  
\end{definition}

By \cite[Proposition 2.9]{Viehweg95}, we see that if $Y$ is projective, then a sheaf $\sG$ is weakly positive on $Y$ if and only if $\sG$ is nef.  The notion of weak positivity, however, generalizes better to quasi-projective varieties.  

The next result closely follows the proof of \cite[Corollary 2.45]{Viehweg95}.  

\begin{lemma}\label{wplem}
Let $g:Y\to B$ be a flat morphism of smooth quasi-projective varieties for which all the fibers are reduced normal varieties with at worst rational singularities, and let $D$ be a divisor on $Y$ such that $\sO_{Y}(D)$ is semi-ample and $\omega_{Y/B}(D)$ is $g$-semi-ample.  Then $g_*\omega_{Y/B}(D)^m$ is weakly positive on $B$ for all $m>0$.  
\end{lemma}
\begin{proof}
For any $\nu$, $\omega_{Y/B}(D)^\nu$ is $g$-semi-ample, and since $\sO_{Y}(D)$ is semi-ample,  $\omega_{Y/B}^{\nu-1}(\nu D)$ is $g$-semi-ample.  By \cite[Theorem 2.40]{Viehweg95} $\omega_{Y/B}(D)^\nu$ is locally free and commutes with arbitrary base change.  

Let $\sH$ be an ample invertible sheaf on $B$, and define 
$$r(\nu)=\min\{\mu>0\;:\;\;g_*\omega_{Y/B}(D)^\nu\otimes \sH^{\mu\nu-1}\textrm{ is weakly positive over $B$}\}$$
In particular, $\omega_{Y/B}(D)^\nu\otimes \sH^{r(\nu)\cdot\nu-1}$ is weakly positive over $B$, and by definition, this means that there's some $\beta>0$ so that 
$$\sym^\beta(g_*\omega_{Y/B}(D)^\nu\otimes \sH^{r(\nu)\cdot\nu-1})\otimes \sH^\beta=\sym^\beta(g_*(\omega_{Y/B}(D)^\nu\otimes g^*\sH^{r(\nu)\cdot \nu}))$$
is globally generated.  

Since $\omega_{Y/B}(D)$ is $g$-semi-ample, we can find an $N$ for which the natural maps $$g^*g_*\omega_{Y/B}(D)^N\to \omega_{Y/B}(D)^N\;\;\;\;\textrm{ and }\;\;\;\; \sym^d(g_*\omega_{Y/B}(D)^N)\to g_*\omega_{Y/B}(D)^{d\cdot N}$$ 
are both surjective, the second map for all $d$.  

Now let $\nu=N$ and $r=r(N)$.  Then by the choice of $r$ and the definition of weakly positive, 
$$\sym^\beta(g_*((\omega_{Y/B}(D)\otimes g^*\sH^r)^N))$$
is globally generated.  Then $\sym^\beta(g^*g_*((\omega_{Y/B}(D)\otimes g^*\sH^r)^N))$ is also globally generated, and by the choice of $N$, we have 
$$\sym^\beta(g^*g_*((\omega_{Y/B}(D)\otimes g^*\sH^r)^N))\epito (\omega_{Y/B)}(D)\otimes g^*\sH^r)^{\beta\cdot N}$$
so $\big(\omega_{Y/B}(D)\otimes g^*\sH^r\big)^{\beta\cdot N}$ is also globally generated, and hence $\omega_{Y/B}(D)\otimes g^*\sH^r$ is semi-ample.  

If we let $\sL=\omega_{Y/B}(D)\otimes g^*\sH^r$, then $\sL^{N-1}\otimes \sO_{Y}(D)$ is also semi-ample, and so if we apply \cite[Theorem 6.16]{Viehweg95} with $\sL_0=\sL^{N-1}\otimes \sO_{Y}(D)$, we get that
$$g_*(\sL_0\otimes \omega_{Y/B})=g_*\big(\omega_{Y/B}(D)^N\big)\otimes \sH^{r(N-1)}$$ is weakly positive over $B$.  But by the choice of $r=r(N)$ as the minimum, we see that we must have 
$$r(N-1)\geq rN-1>(r-1)N -1$$ which simplifies to $r-1<N$ or $r\leq N$.  

Thus, setting $\mu=N$, we get that $g_*\big(\omega_{Y/B}(D)^N\big)\otimes \sH^{(N^2-1)}$ is weakly positive.  Note that the choice of $N$ is based on the behavior of $\omega_{Y/B}(D)$ along the fibers of $g$, and so the previous statement holds if we replace $B$ by $B^\prime$ for any nonsingular finite cover $\tau:B^\prime\to B$ and any ample invertible sheaf $\sH$ on $B^\prime$.  Hence, by applying \cite[Lemma 2.15(3)]{Viehweg95}, we see that $g_*(\omega_{Y/B}(D)^N)$ is weakly positive.  

Applying a similar argument, we can find a $\beta>0$ so that $\sym^\beta(g_*(\omega_{Y/B}(D)^N))\otimes\sH^{\beta\cdot N}$ is globally generated, which means that $\omega_{Y/B}(D)\otimes\sH$ is semi-ample.  This means that $\big(\omega_{Y/B}(D)\otimes\sH\big)^{m-1}\otimes \sO_{Y}(D)$ is also semi-ample, and applying \cite[Theorem 6.16]{Viehweg95} again, $$g_*\big(\omega_{Y/B}^{m-1}(mD)\otimes g^*\sH^{m-1}\otimes \omega_{Y/B}\big)=g_*\omega_{Y/B}(D)^m\otimes \sH^{m-1}$$ is weakly positive.  Once again, the above statement holds if we replace $B$ by $B^\prime$ for a finite cover $\tau:B^\prime\to B_0$ and for any ample invertible sheaf $\sH$ on $B^\prime$.  Thus by applying \cite[Lemma 2.15(3)]{Viehweg95} again, we get that $g_*\omega_{Y/B}(D)^m$ is weakly positive for all $m>0$.  
\end{proof}

\begin{cor}\label{s-a}  Under the conditions of Lemma \ref{wplem}, assume also that $\omega_{Y/B}(D)$ is $g$-ample and that $B\iso \P^1$. Then $\omega_{Y/B}(D)$ is semi-ample. 
\end{cor}
\begin{proof}
Let $\sL=\omega_{Y/B}(D)$.  For any $b\in B$, the fiber $Y_b$ is a divisor so $\sI_{Y_b}\iso \sO_{Y}(-Y_b)\iso g^*\sO_{B}(-b)$, and we can consider the short exact sequence $$0\to g^*\sO_{B}(-b)\to \sO_Y\to \sO_{Y_b}\to 0$$

Twisting by $\sL^N$ and pushing forward by $g$, we get the sequence 
$$0\to g_*\sL^N\otimes \sO_B(-b) \to g_*\sL^N\to g_*(\sL^N|_{Y_b}) \to R^1g_*\sL^N\otimes \sO_B(-b)$$

Since $\sL=\omega_{Y/B}(D)$ is $g$-ample, by choosing $N$ large enough, we have that $R^1g_*\sL^N=0$ by \cite[Theorem III, 12.9 and III, 5.2(b)]{Hartshorne77}. 
So we have the short exact sequence 
$$0\to g_*\sL^N\otimes \sO_B(-b) \to g_*\sL^N\to g_*(\sL^N)|_{Y_b}) \to 0$$
The sheaf $g_*(\sL^N|_{Y_b})$ is just $H^0(Y_b, \sL^N|_{Y_b})$, so when we take global sections of the above sequence, we get 

$$H^0(B, g_*\sL^N\otimes \sO_B(-b))\to H^0(B, g_*\sL^N) \to H^0(Y_b, \sL^N|_{Y_b}) \to H^1(B, g_*\sL^N(-b))$$
Since $B\iso\P^1$, then $g_*\sL^N$ splits as a direct sum of line bundles so $g_*\sL^N\iso \bigoplus_i \sO_{\P^1}(a_i)$.  Since $g_*\sL^N$ is weakly positive by \ref{wplem}, $a_i\geq 0$ for all $i$.  Thus 
$$H^1(B, g_*\sL^N\otimes \sO_B(-b))\iso \bigoplus_i H^1(\P^1, \sO_{\P^1}(a_i-1))\iso \bigoplus_i H^0(\P^1, \sO_{\P^1}(-1-a_i))=0$$
So we have a surjective map $H^0(B, g_*\sL^N) \twoheadrightarrow H^0(Y_b, \sL^N|_{Y_b})$.  Since $\sL$ is $g$-ample, for a large enough choice of $N$, $\sL^N|_{Y_b}$ is globally generated.  Since $H^0(Y,\sL^N)\iso H^0(B, g_*\sL^N)\twoheadrightarrow H^0(Y_b, \sL^N|_{Y_b})$, and $\sL^N|_{Y_b}$ is globally generated for every $b\in B$, $\sL^N$ is globally generated over $Y$ for an appropriately large $N$.  Thus $\sL=\omega_{Y/B}(D)$ is semi-ample.  

\end{proof}

The proof of the next corollary closely follows the proof of \cite[Theorem 6.22]{Viehweg95}.

\begin{cor}\label{ample}
Let $g:(Y,D)\to B$ be a non-isotrivial family with $\Delta\subset B$ the locus of $D$-singular fibers.  Suppose that $\sO_Y(D)$ is semi-ample, and that the line bundle $\omega_{Y/B}(D)$ is $g$-ample.  Suppose further that for $b\in B$, the image of $Y_b$ under the embedding via $\omega_{Y_b}(D_b)^k$ into some $\P^n$ has only finitely many automorphisms that can be written as the restriction of a linear automorphism $\sigma$ of $\P^n$ to the image of $Y_b$.  Then $g_*\omega_{Y/B}(D)^m$ is ample for $m>0$ appropriately large.  
\end{cor}

\begin{proof}
Theorem \ref{det} says that we can choose an $n$ be such that $\det \big(g_*\omega_{Y/B}(D)^n\big)$ is ample.  Let $r=\rank g_*\omega_{Y/B}(D)^n$,  $Y^r$ the $r$-fold product of $Y$ over $B$, and $\PR_i:Y^r\to Y$ the $i$th projection.   Now consider the map $g^r:Y^r\to B$.  The induced morphism $g^r$ is again flat and we have that $\omega_{Y^r/B}=\bigotimes_{i=1}^{r} \PR_i^*\omega_{Y/B}$.  Thus if we let $D^\star=\sum_{i=1}^r \PR^*_iD$, we get that $$\omega_{Y^r/B}(D^\star)=\bigotimes_{i=1}^{r} \PR_i^*\omega_{Y/B}(D)$$
By flat base change 
$$g^r_*\omega_{Y^r/B}(D^\star)^n=\bigotimes_{i=1}^{r} g_*\omega_{Y/B}(D)^n$$
Now, we have the natural inclusion $$\det g_*\omega_{Y/B}(D)^n\to \bigotimes^r g_*\omega_{Y/B}(D)^n=g^r_*\omega_{Y^r/B}(D^\star)^n$$
which splits locally, so the zero divisor $\Gamma$ of the induced section 
$$\sO_Y\to ({g^r}^*\det g_*\omega_{Y/B}(D)^n)^{-1}\otimes \omega_{Y^r/B}(D^\star)^n$$ does not contain a fiber of $g^r$.  

Let $\sL_0=\omega_{Y^r/B}(D^\star)^{(m-1)}\otimes \sO_{Y^r}(D^\star)$ for some $m>0$.  We can see that since $\omega_{Y/B}(D)$ and $\sO_Y(D)$ are both semi-ample, so is $\omega_{Y^r/B}(D^\star)^{(m-2)}\otimes \sO_{Y^r}(2D^\star)$.  Thus by \cite[Theorem 6.16]{Viehweg95}, $g^r_*(\sL_0)$ is weakly positive over $B$.  
Notice that $$\begin{array}{rl} \sL_0^n & = \omega_{Y^r/B}(D^\star)^{n(m-1)}\otimes \sO_{Y^r}(nD^\star)\\
& = {g^r}^*(\det g_*\omega_{Y/B}(D)^n)^{m-1}\otimes \sO_{Y^r}((m-1)\Gamma + nD^\star)\\
\end{array}$$
We can now apply \cite[Theorem 6.21]{Viehweg95} to get that for some large $e>0$ 
$$\sym^e(\bigotimes^r g_*\omega_{Y/B}(D)^m)\otimes (\det g_*\omega_{Y/B}(D)^n)^{-(m-1)}$$ is weakly positive over $B$.  Thus if we pick $m$ such that $m-1$ is divisible by $r$, \cite[Lemma 2.25]{Viehweg95} and \cite[Lemma 2.24]{Viehweg95} give us that $g_*\omega_{Y/B}(D)^m$ is ample.  
\end{proof}

\section{Vanishing Theorem}
The second major ingredient that we will use in the proof of Theorem \ref{main} is a vanishing theorem.  However, for the proof of the following vanishing theorem, we will need to restrict ourselves to considering families with reduced fibers.  We also need that the divisor $D+g^*\Delta$ is simple normal crossing.  Both of these added hypotheses mimic the requirements of \cite[Lemma 1.1]{Kovacs00a}, from which the proof of Theorem \ref{van} borrows much machinery.  

% Big Vanishing Theorem

\begin{theorem}\label{van}
Let $g:Y\to B$ be a morphism from a smooth $n$-dimensional projective variety $Y$ to a smooth projective curve $B$ with reduced fibers. Let $D$ be an effective snc divisor on $Y$ that is horizontal over $B$.  Let $\Delta\subset B$ be the locus of $D$-singular fibers, and suppose that $D+g^*\Delta$ is an snc divisor. Further, let $\phi:Y \to X$ be a morphism to an $n$-dimensional projective variety $X$ and $f:X\to B$ a morphism such that $g=f\circ \phi$ and assume that $\phi$ is an isomorphism away from $g^{-1}(\Delta)$. Let $\sL$ be a line bundle on $Y$ such that there exists an ample line bundle $\sA$ on $X$ and a natural number $m\in\N$ such that $\sL^m\iso \phi^*\sA$. Assume finally that $\omega_B(\Delta)^{-1}$ is nef.  Then for any line bundle, $\sK$, such that $\sK\supseteq\sL(-D)$, $$H^n(Y,\sK\otimes g^*\omega_B)=0$$
\end{theorem}

\begin{proof}

Since $D+g^*\Delta$ is a snc divisor, we have a short exact sequence of locally free sheaves by Lemma \ref{locfree}:  
$$0\to g^*\omega_B(\Delta)\to \Omega_Y\big(\log (D+g^*\Delta)\big) \to \Omega_{Y/B}\big(\log (D+g^*\Delta)\big)\to 0$$

By \cite[II Ex. 5.16]{Hartshorne77}, we see that taking determinants gives us
$$
\begin{array}{rl} \Omega^{n-1}_{Y/B}\big(\log (D+g^*\Delta)\big) &
 =\det\Omega_{Y/B}\big(\log (D+g^*\Delta)\big) \\
& =\det\Omega_Y\big(\log (D+g^*\Delta)\big)\otimes g^*\omega_B(\Delta)^{-1}\\
& =\omega_Y(D+g^*\Delta)\otimes g^*\omega_B(\Delta)^{-1}\\
& =\omega_{Y/B}(D)\\
\end{array}$$
since $D+g^*\Delta$ is reduced.  

Also by \cite[II Ex. 5.16]{Hartshorne77}, for all $p=0,\ldots,n-1$ we have the short exact sequences
$$0\to \Omega^{p-1}_{Y/B}\big(\log (D+g^*\Delta)\big)\otimes g^*\omega_B(\Delta) \to \Omega^p_Y\big(\log (D+g^*\Delta)\big) \to \Omega^p_{Y/B}\big(\log (D+g^*\Delta)\big)\to 0$$

Let $\sL_p=\sL\otimes g^*\omega_B(\Delta)^{p-(n-1)}$ for $p=0, \ldots, n-1$.  Twisting the above sequence by $\sL_p^{-1}$, we get 
$$0\to \Omega^{p-1}_{Y/B}\big(\log (D+g^*\Delta)\big)\otimes \sL_{p-1}^{-1}\to 
\Omega^p_Y\big(\log (D+g^*\Delta)\big)\otimes \sL_p^{-1} \to 
\Omega^p_{Y/B}\big(\log (D+g^*\Delta)\big)\otimes \sL_p^{-1}\to 0$$

Now, $\sL_p^m=\sL^m\otimes g^*\omega_B(\Delta)^{m(p-(n-1))}=\phi^*(\sA\otimes f^*\omega_B(\Delta)^{m(p-(n-1))})$.  Since $\omega_B(\Delta)^{-1}$
is nef and $\sA$ is ample, $\sA\otimes f^*\omega_B(\Delta)^{m(p-(n-1))}$ is also ample.  Since $\phi$ is an isomorphism outside of $g^{-1}(\Delta)$, we can apply \cite[Corollary 6.7]{EV92} to get that $$H^{n-1-p}\big(Y, \Omega^p_Y\big(\log (D+g^*\Delta)\big)\otimes \sL_p^{-1}\big)=0$$

Thus the map we get from the corresponding long exact sequence
$$H^{n-1-p}\big(Y, \Omega^p_{Y/B}\big(\log (D+g^*\Delta)\big)\otimes \sL_p^{-1}\big)
\to H^{n-1-(p-1)}\big(Y, \Omega^{p-1}_{Y/B}\big(\log (D+g^*\Delta)\big)\otimes \sL_{p-1}^{-1}\big)$$
is injective for all $p$.  Thus by composing these maps, we get that
$$H^{0}\big(Y, \Omega^{n-1}_{Y/B}\big(\log (D+g^*\Delta)\big)\otimes \sL_{n-1}^{-1}\big)\to
H^{n-1}\big(Y, \Omega^{0}_{Y/B}\big(\log (D+g^*\Delta)\big)\otimes \sL_{0}^{-1}\big)$$ 
is injective, that is
$$H^{0}(Y, \omega_{Y/B}(D)\otimes \sL^{-1})\to
H^{n-1}(Y, \sL_{0}^{-1})$$
is injective.

But $H^{n-1}(Y, \sL_{0}^{-1})=0$ by \cite[6.7]{EV92}, so $H^{0}(Y, \omega_{Y/B}(D)\otimes \sL^{-1})=0$, and then Serre duality, 
$$H^{n}(Y, \sL(-D)\otimes g^*\omega_B)=0$$

Now if $\sK$ is a line bundle with $\sL(-D)\subseteq \sK$, and we call $\sQ=\sK/\sL(-D)$, then $\sQ$ must be supported on a proper subvariety.  Thus $H^n(Y, \sQ\otimes g^*\omega_B)=0$, so the map $$H^n(Y, \sL(-D)\otimes g^*\omega_B)\to H^n(Y, \sK\otimes g^*\omega_B)$$ is surjective.  Thus $H^n(Y, \sK\otimes g^*\omega_B)=0$.  

\end{proof}

We won't use the full generality of this vanishing theorem in the proof of Theorem \ref{main}.  We'll make the reduction that $\phi=\id$ and choose $\sL$ to be an ample line bundle.

\section{Main Result}
In order to take advantage of Theorem \ref{van} from the previous section, we must have a family for which $D+g^*\Delta$ is snc.  

Given a family $g:(Y,\sL)\to B$ with $\sL\iso \omega_Y(D)$ for an effective snc divisor $D$, we can replace it with another family, isomorphic to the original one over $B_0=B\setminus\Delta$ which has $D+g^*\Delta$ is snc, since we are interested ultimately in the size of the set $\Delta$ from Definition \ref{delta}.  We do this by taking a resolution  $\sigma:\tilde{Y}\to Y$ where the exceptional locus lies in $g^{-1}\Delta$, and we can then replace $Y$ with $\tilde{Y}$, $g$ with $\tilde{g}=\sigma\circ g$, and $D$ with its strict tansform $\tilde{D}$, and we'll have a family that's isomorphic to the original one over $B_0$. 
 
In particular, we can take an embedded resolution of the divisor $D+g^*\Delta$, and since $D$ itself was snc, the exceptional locus will lie in $g^{-1}\Delta$.  Then the support of the replacement divisor $\tilde{D}+\tilde{g}^*\Delta$ will be snc.  However, $\tilde{D}+\tilde{g}^*\Delta$ itself must be snc, so in particular, it must be reduced.  Thus the resolution $\sigma$ cannot introduce multiple fibers.  

We can ensure this in the following way.  First, we must assume that the fibers of $g$ which are not smooth are at least snc, so that multiple fibers are not introduced from resolving singularities in the fibers.  Secondly, if we require that each component of the divisor $D$ intersects the fibers of $g$ transversely and avoids the singular locus of the fibers over $\Delta$, we will avoid introducing multiple fibers in resolving the singularities of $D+g^*\Delta$.  These second conditions ensure for each succesive blow-up in the resolution, the exceptional locus will lie away from the singular set of the fiber.  Thus the resolution may introduce new components to the fiber, but each new components is introduced with multiplicity $1$. 

Now that we have Corollary \ref{ample} and Theorem \ref{van}, we can use the structure of the proof of \cite[0.2]{Kovacs00a} to prove the main result, Theorem \ref{main}, which we now restate.  

\begin{theorem}\label{main2}
Let $g:(Y,D)\to \P^1$ be a family, $\Delta\subset \P^1$ the locus of $D$-singular fibers and $B_0=\P^1\setminus \Delta$.  Suppose that the line bundle $\omega_{Y/B}(D)$ is $g$-ample and that $|\Aut(Y_b)|<\infty$ for $b\in \P^1$.  Suppose also that $D$ is an snc divisor for which each component intersects the fibers of $g$ transversely and avoids the singular locus of the fibers over $\Delta$.  

Then either $g$ is isotrivial or $|\Delta|\geq 3$.  
\end{theorem}

\begin{proof}
The hypothesis on the divisor $D$ allows us to take an embedded resolution of $D+g^*\Delta$ in such a way that the resulting divisor $\tilde{D}+\tilde{g}^*\Delta$ not only has snc support but is also reduced, as describe above.  Thus, by potentially replacing $g$ with $\tilde{g}$, we may assume that $D+g^*\Delta$ is snc.  

Now let us assume that $g$ is non-isotrivial and that $|\Delta|\leq 2$.  

By Corollary \ref{ample}, $g_*\omega_{Y/\P^1}(D)^m$ is ample for some $m>0$.  Thus 
$$g_*\omega_{Y/\P^1}(D)^m\iso \bigoplus_{i}\sO_{\P^1}(a_i), \;\;\;\;\;\;a_i\geq 1$$
Let $s,t\in B$, and let $\sI_{s,t}$ be their ideal sheaf.  Then there exist an $l_0\in \N$ such that for every $l\geq l_0$ and $i>0$, 
$$H^i(\P^1, \sym^l(g_*\omega_{Y/\P^1}(D)^m)\otimes \sI_{s,t})=0$$

Thus if we take the short exact sequence $$0\to \sI_{s,t}\to \sO_{\P^1}\to \sO_{\{s,t\}}\to 0,$$ twist by $g_*\omega_{Y/\P^1}(D)^m$, and then apply the long exact sequence in cohomology, we get that the map
$$\nu: H^0(\P^1, \sym^l(g_*\omega_{Y/\P^1}(D)^m))\to \left(\sym^l(g_*\omega_{Y/\P^1}(D)^m)\otimes k(s)\right)\oplus \left(\sym^l(g_*\omega_{Y/\P^1}(D)^m)\otimes k(t)\right)$$ 
is surjective.  

Additionally, since $\omega_{Y/\P^1}(D)^m$ is semi-ample when restricted to $Y_s$ and $Y_t$, the map 
$$\epsilon: \sym^l(g_*\omega_{Y/\P^1}(D)^m)\to g_*\omega_{Y/\P^1}(D)^{lm}$$
is also surjective over $B_0$ for $l>>0$.  This gives us the following commutative diagram:

$$\xymatrix{H^0(\P^1, \sym^l(g_*\omega_{Y/\P^1}(D)^m)) \ar[r]^-{\nu} \ar[d] & \left(\sym^l(g_*\omega_{Y/\P^1}(D)^m)\otimes k(s)\right)\oplus \left(\sym^l(g_*\omega_{Y/\P^1}(D)^m)\otimes k(t)\right) \ar[d]_{\epsilon}\\
H^0(\P^1, \sym^l(g_*\omega_{Y/\P^1}(D)^m)) \ar[r]_-{\sigma}  & \left(g_*\omega_{Y/\P^1}(D)^{lm}\otimes k(s)\right)\oplus \left(g_*\omega_{Y/\P^1}(D)^{lm}\otimes k(t)\right)\\
}
$$

Since $\nu$ and $\epsilon$ are both surjective, $\sigma$ must be surjective as well.  

Replacing $lm$ by $r$, we get that

\begin{equation}\label{sepfibers}
H^0(Y, \omega_{Y/\P^1}(D)^r)\to H^0(Y_s, \omega_{Y_s}(D_s)^r)\oplus H^0(Y_t, \omega_{Y_t}(D_t)^r)
\end{equation}
is surjective for sufficiently large and divisible $r>0$.  

Choose $r$ large enough so that \ref{sepfibers} is surjective and so that $\omega_{Y_s}(D_s)^r$ is ample.  Then $\omega_{Y/\P^1}(D)^r$ is generated by global sections, and it defines a rational map $$\phi=\phi_{\omega_{Y/\P^1}(D)^r}:Y\ratmap X,$$
where $\phi|_{Y_s}=\phi_{\omega_{Y_s}(D_s)^r}$ and $\phi$ separates the fibers of $g$.  Since $\omega_{Y/\P^1}(D)^r$ is also ample on the fibers, the map $\phi$ is one-to-one, and thus some power of $\omega_{Y/\P^1}(D)^r$ defines an embedding.  Thus $\omega_{Y/\P^1}(D)$ is ample on $Y$.  

Since $|\Delta|\leq 2$, $\omega_{\P^1}(\Delta)^{-1}$ is nef, so we can apply Theorem \ref{van} with $\phi=\id$, $\sL=\omega_{Y/\P^1}(D)$ and $\sK=\omega_{Y/\P^1}=\sL(-D)$ to get that $$H^n(Y, \omega_{Y/\P^1}\otimes g^*\omega_{\P^1})=H^n(Y, \omega_Y)=0$$ 
However, since $Y$ is projective, Serre duality implies that $H^0(Y, \sO_Y)=0$, which is a contradiction. 

Thus either $g$ is isotrivial or $\omega_{\P^1}(\Delta)^{-1}$ is not nef, that is, $|\Delta|\geq 3$.

\end{proof}

\bibliographystyle{skalpha}
\bibliography{ref2}

\end{document}